\documentclass[twoside]{amsart}
\usepackage{amsmath,amssymb}
\usepackage{amsaddr}
\usepackage{amsthm}
\usepackage{tikz}
\tikzset{cross/.style={preaction={-,draw=white,line width=6pt}}}
\usepackage{amsfonts,mathrsfs}
\usepackage{color}
\usepackage{graphicx}
\usepackage{datetime2}
\usepackage{refcount}
\usepackage{tcolorbox}
\tcbuselibrary{breakable}
\tcbuselibrary{skins}
\tcbset{enhanced,breakable}
\usepackage{bookmark}
\usepackage{xurl}
\hypersetup{unicode,bookmarksnumbered=true,hidelinks,final}
\usepackage{ulem}
\numberwithin{equation}{section}
\allowdisplaybreaks
\everymath{\displaystyle}
\usepackage{fancyhdr}
\theoremstyle{definition}
\newtheorem{definition}{Definition}[section]
\newtheorem{proposition}[definition]{Proposition}

\newtheorem{remark}[definition]{Remark}
\newtheorem{note}[definition]{Note}
\newtheorem{example}[definition]{Example}
\newtheorem{fact}[definition]{Fact}

\newtheorem*{proof*}{Proof}
\newcounter{maintheoremcounter}
\newtheorem{maintheorem}[maintheoremcounter]{Theorem}

\newcommand{\pd}[2]{\frac{\partial #1}{\partial #2}}

\newcommand{\Rn}[1]{\mathbb{R}^{#1}}
\newcommand{\Cn}[1]{\mathbb{C}^{#1}}

\newcommand{\id}{\mathrm{id}}
\newcommand{\iu}{\sqrt{-1}}

\newcommand{\Kahlerian }{K\"{a}hlerian}

\newcommand{\HesBor}{1}

\newcommand{\ThmHesBor}{
    \setcounter{maintheoremcounter}{\HesBor}
    \addtocounter{maintheoremcounter}{-1}
    \begin{maintheorem}
        (1)~For~$(M,\nabla)$,~let~$(I,J,K)$ be the almost para-quaternionic
        structure naturally induced on the tangent bundle~$TM$. 
        Then the following conditions are equivalent.\\
        \quad (1-a)~$\nabla$ is flat.\\
        \quad (1-b)~$I,J$ and~$K$ are individually integrable. \\
        \quad (1-c)~$(I,J,K)$ is integrable.\\
        (2)~For~$(M,\nabla,g)$,~let~$(I,J,K,h,k,\omega)$ be the almost Born
        structure naturally induced on the tangent bundle~$TM$. 
        Then the following conditions are equivalent.\\
        \quad (2-a)~$(\nabla, g)$ is a Hessian structure.\\
        \quad (2-b)~$(I,J,K,h,k,\omega)$ is integrable.\\
        \quad (2-c)~$(I,J,K,h,k,\omega)$ is strongly integrable.
    \end{maintheorem}
}

\title{A Born Structure on the Tangent Bundle of a Hessian Manifold}
\author{Hakobi Sakamoto}
\email{skmtha@ms.u-tokyo.ac.jp}
\date{}
\begin{document}
\begin{abstract}
    The Hessian structure, introduced in \cite{Shi76},
    is a geometric structure consisting of a pair~$(\nabla,g)$
    of an affine connection~$\nabla$
    and a Riemannian metric~$g$ satisfying certain conditions.
    On the other hand, the Born structure, introduced in \cite{FLM14}, is
    a strictly stronger geometric structure  
    than an almost (para-)Hermitian structure.
    In \cite{MS19}, it is proved that 
    for a given manifold endowed with a pair~$(\nabla, g)$, one can introduce
    an almost Born structure on the tangent bundle.

    In this article, we study the equivalence between
    the conditions that the pair~$(\nabla, g)$ defines a Hessian structure,
    and that the induced almost Born structure is integrable.
\end{abstract}
\maketitle
\thispagestyle{fancy}

\tableofcontents
\clearpage


\section{Introduction}
A \textit{Hessian manifold} is a manifold equipped with a flat affine 
connection~$\nabla$ and a Riemannian metric~$g$ such that 
the~$(0,3)$ tensor field~$\nabla g$ is symmetric in all three
variables. It was defined in \cite{Shi76} and 
is considered to be the real analogue of \Kahlerian~manifolds. 
Hessian manifolds are also known as \textit{dually flat manifolds} in 
\textit{information geometry}. They are studied as ``natural'' geometric structures
introduced on a class of statistical models known as \textit{exponential families}
(such as the normal, beta and categorical distributions).

The \textit{Born structure} is a geometric structure introduced in
\cite{FLM14} as part of the geometrical foundation of \textit{double field theory},
a theoretical framework in physics motivated by string theory.
An \textit{almost Born structure}~$(I,J,K,h,k,\omega)$ on a manifold
consists of $(1,1)$-tensor fields $I,J$ and~$K$ satisfying 
$-I^2=J^2=K^2=\id$, a Riemannian metric~$h$, 
a pseudo-Riemannian metric~$k$ of signature~$(n,n)$, 
and a non-degenerate 2-form~$\omega$, with certain compatibility conditions.
It includes other geometric structures; for example,
$(I,h,\omega)$ defines an almost Hermitian structure.
Thus, a Born structure is a strictly stronger 
than each of these individually.

In \cite{Dom62},\cite{CFG96} or \cite{MS19},
it is proved that for a given manifold~$M$, an affine connection~$\nabla$ on~$M$,
and a Riemannian metric~$g$ on~$M$, one can introduce almost (para-)complex 
structures or metrics on the total space of 
the tangent bundle~$TM$. This yields an almost (para-)Hermitian structure 
or an almost Born structure on~$TM$.
It is also proved in \cite{Sat07} that the introduced almost 
Hermitian structure~$(I,h,\omega)$ on~$TM$ is \Kahlerian~if
and only if the pair~$(\nabla,g)$ defines a Hessian structure on~$M$.

In this paper, we study the same equivalence for the 
introduced almost Born structure on~$TM$.
Specifically, we define two types of integrability for
an almost Born structure: \textit{integrability} and 
\textit{strong integrability}.
We then show that the introduced almost Born structure
satisfies either one of these integrability conditions if and only if
the pair~$(\nabla,g)$ defines a Hessian structure on~$M$.
The main result is as follows:

\ThmHesBor

We note that for general almost Born structures, beyond those introduced
on tangent bundles, the strong integrability is a strictly stronger condition
than the integrability.

\section{Preliminaries}

\subsection{Hessian structures}
In this section, we review some fundamental results about Hessian manifolds.\\
In the following, let $M$ be a manifold, $\nabla$ an affine connection on~$M$, 
and~$g$ a Riemannian metric on~$M$.
\begin{definition}[\cite{Shi76}]
    A pair~$(\nabla,g)$ is said to be a \textit{Hessian structure} on~$M$ if
    $\nabla$ is flat i.e., its curvature and torsion tensors vanish, 
    and (0,3)-tensor field
    $\nabla g$ is symmetric in the three variables.
\end{definition}
\begin{definition}[\cite{Ama85}]
    There exists a unique affine connection~$\nabla^*$ for a given~$\nabla$ 
    and~$g$ satisfying the condition
    \begin{equation*}
        Xg(Y,Z)=g(\nabla_XY,Z)+g(Y,\nabla^*_XZ)
    \end{equation*}
    for any vector fields~$X,Y$ and~$Z$, which is called the \textit{dual connection}
    of $\nabla$ with respect to $g$.
\end{definition}
\begin{fact}[\cite{Fuj21}]\label{not:HesTor}~\\
    (1)~The conditions~$R=0$ and~$R^*=0$
    are equivalent for the curvature tensors~$R$ and~$R^*$ of~$\nabla$ 
    and~$\nabla^*$, respectively.\\
    (2)~If any two of the following four conditions hold, then so do
    the remaining two.\\
    \quad (2-a)~$\nabla$ is torsion-free.\\
    \quad (2-b)~$\nabla^*$ is torsion-free.\\
    \quad (2-c)~$\nabla g$ is symmetric in the three variables.\\
    \quad (2-d)~$\frac{1}{2}(\nabla+\nabla^*)$ coincides with the Levi-Civita
    connection of~$g$.
\end{fact}

\begin{example}\label{exm:StaHes}
    The Euclidean space~$\Rn{n}$ admits the natural Hessian structure
    $(\nabla, g)$, where~$\nabla$ denotes the flat affine connection
    induced by the linear structure, and~$g$ is the Euclidean metric.
\end{example}

\subsection{Complex and para-complex structures}
In this section, we review the definitions of almost complex structures and
almost para-complex structures.\\
In the following, let~$N$ be a $2n$-dimensional manifold, 
and~$TN,T^*N$~the tangent and cotangent bundle of~$N$. 
In this paper, we naturally identify the endomorphism bundle~$\mathrm{End}(TN)$
and the (1,1)-tensor bundle~$TN\otimes T^*N.$ 

\begin{definition}[\cite{CFG96}]~\\
    (1)~A (1,1)-tensor field~$I$ on~$N$
    is called an \textit{almost complex structure} if it satisfies~$I^2=-\id.$\\
    (2)~A (1,1)-tensor field~$K$ on~$N$
    is called an \textit{almost para-complex structure} if it satisfies~$K^2=\id$,
    and the dimensions of the ($\pm 1$)-eigenspaces at each point of~$N$ 
    are both equal to $n$.
\end{definition}

\subsection{Para-quaternionic and Born structures}
In this section, we review the definitions of almost para-quaternionic structures
and almost Born structures.\\
In the following, let~$N$ be a $2n$-dimensional manifold.
\begin{definition}[\cite{IZ05}]
    Let $I,J$ and~$K$ be (1,1)-tensor fields
    on~$N$.
    The triple~$(I,J,K)$ is called an 
    \textit{almost para-quaternionic structure} on~$N$ 
    if it satisfies~$-I^2=J^2=K^2=\id$ and~$IJK=-\id.$
\end{definition}
\begin{remark}
    An almost para-quaternionic structure is also called 
    a quaternion structure of the second kind or a complex product structure
    (\cite{YA73}~and~\cite{And05}).
\end{remark}
\begin{note}
    We note that if~$(I,J,K)$ is an almost para-quaternionic structure, 
    then by definition $I$ is
    an almost complex structure and $J,K$ are almost para-complex structures.
    In addition, $I,J$ and~$K$ satisfy the following anti-commutation relations.
    \begin{equation*}
        I=JK=-KJ,\quad  J=-KI=IK,\quad  K=-IJ=JI.
    \end{equation*}
\end{note}

\begin{definition}[\cite{FRS17} and \cite{FRS19}]
    Let~$(I,J,K)$ be an almost para-quaternionic structure,~$h$ a Riemannian metric,
    $k$ a pseudo-Riemannian metric of signature~$(n,n)$,
    and $\omega$ a non-degenerate 2-form. 
    The tuple~$(I,J,K,h,k,\omega)$ is called an \textit{almost Born structure} on~$N$
    if it satisfies the following identities:
    \begin{equation*}
        I=h^{-1}\circ \omega,\quad  J=k^{-1}\circ h,\quad  K=\omega^{-1}\circ k.
    \end{equation*}
    We note that the tensor fields~$h, k, \omega$ define
    the natural isomorphisms~$TN\rightarrow T^*N$ via~$X\mapsto h(X,\cdot)$,
    and similarly for~$k$ and~$\omega$.
\end{definition}
\begin{remark}
    If $(I,J,K,h,k,\omega)$ is an almost Born structure, then the
    following diagram commutes:
    \center
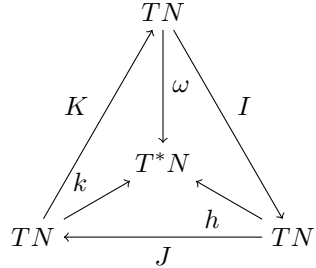
    \begin{tikzpicture}[auto]
        \node (a) at (0, 0) {$T^*N$};
        \node (x) at (0, 2) {$TN$};
        \node (y) at (1.73,-1) {$TN$};
        \node (z) at (-1.73,-1) {$TN$};
        \draw[->] (x) to node {$\omega$} (a);
        \draw[->] (y) to node {$h$} (a);
        \draw[->] (z) to node {$k$} (a);
        \draw[->] (x) to node {$I$} (y);
        \draw[->] (y) to node {$J$} (z);
        \draw[->] (z) to node {$K$} (x);
    \end{tikzpicture}
    \flushleft
\end{remark}

\subsection{Integrability of para-quaternionic and Born structures}
In this section, we define integrability conditions on 
almost para-quaternionic structures and almost Born structures.\\
In the following, let~$N$ be a~$2n$-dimensional manifold.
\begin{definition}[\cite{And05}~and~\cite{HKP25}]
    For a (1,1)-tensor field~$A$ on~$N$, 
    we define the \textit{Nijenhuis tensor} $N_A$ 
    as
    \begin{equation*}
        N_A(X,Y):=A^2[X,Y]-A([AX,Y]+[X,AY])+[AX,AY]
    \end{equation*}
    for any vector fields~$X$ and~$Y$.
\end{definition}
\begin{remark}
    The above definition differs from those in 
    \cite{YA73},\cite{FRS17} and \cite{FRS19}
    up to scalar multiplication.
\end{remark}
\begin{fact}[Newlander-Nirenberg theorem]\label{fac:alcom}
    For an almost complex structure~$I$ on~$N$, 
    the following two conditions are equivalent.\\
    (1)~The Nijenhuis tensor~$N_I$ vanishes.\\
    (2)~Locally, there exists a chart in which
    $I=\begin{pmatrix}
        0& -1_n\\ 1_n&0
    \end{pmatrix}$.
\end{fact}
\begin{definition}[\cite{Sil08}]
    An almost complex structure is said to be \textit{integrable} if 
    the two equivalent conditions described above are satisfied.
\end{definition}
\begin{fact}[\cite{CMMS04}]\label{fac:alpcom}
    For an almost para-complex structure~$K$ on~$N$,
    the following three conditions are equivalent.\\
    (1)~The Nijenhuis tensor~$N_K$ vanishes.\\
    (2)~The tangent distributions, given by
    the~$\pm 1$-eigenspaces of~$K$ at each point, are 
    both integrable.\\
    (3)~Locally, there exists a chart in which
    $K=\begin{pmatrix}
        1_n&0\\0&-1_n
    \end{pmatrix}$.
\end{fact}
\begin{definition}[\cite{CMMS04}]
    An almost para-complex structure is said to be \textit{integrable} if 
    the three equivalent conditions described above are satisfied.
\end{definition}

\begin{definition}[\cite{YA73}]
    An almost para-quaternionic structure~$(I,J,K)$
    is said to be \textit{integrable} if 
    locally there exists a chart in which simultaneously
    $I=\begin{pmatrix}
        0& -1_n\\ 1_n&0
    \end{pmatrix},
    J=\begin{pmatrix}
        0& 1_n\\ 1_n&0
    \end{pmatrix},$ and
    $K=\begin{pmatrix}
        1_n&0\\0&-1_n
    \end{pmatrix}$.
\end{definition}
\begin{definition}
    Let $(I,J,K,h,k,\omega)$ be an almost Born structure.\\
    (1)~The structure~$(I,J,K,h,k,\omega)$ is said to be \textit{integrable}
    if the (1,1)-tensor fields $I,J$ and~$K$ are individually integrable
    and $d\omega=0$.\\
    (2)~The structure~$(I,J,K,h,k,\omega)$ is said to be \textit{strongly integrable}
    if the almost para-quaternionic structure~$(I,J,K)$ is integrable
    and $d\omega=0$.
\end{definition}
\begin{remark}
    In \cite{HKP25}, an almost Born structure is said to be integrable
    under the same condition as in this paper.
    Furthermore, the above condition (2) is strictly stronger than (1).
\end{remark}
In fact, the following holds.
\begin{note}\label{not:PQint}
    Consider the following two conditions on an 
    almost para-quaternionic structure~$(I,J,K)$. Then 
    the condition (2) is strictly stronger than (1).\\
    (1)~The (1,1)-tensor fields~$I,J$ and~$K$
    are individually integrable.\\
    (2)~The almost para-quaternionic structure is integrable.
\end{note}
\begin{proof}
    The implication $(2)\Rightarrow(1)$ follows from Facts \ref{fac:alcom} and
    \ref{fac:alpcom}. See \cite{And05} for examples where the condition 
    (1) holds but (2) does not.
\end{proof}

We present a standard example below.
\begin{example}\label{exm:StaBor}
    There exists a standard strongly integrable Born structure on~$\Cn{n}$.
    Let~$(z^1,\ldots, z^n)$ be the standard holomorphic chart on~$\Cn{n}$, and 
    $z^i=x^i+\iu y^i$. We construct a strongly integrable Born structure
    $(I,J,K,h,k,\omega)$ as follows:
    \begin{align*}
        I&:=\pd{}{y^i}\otimes dx^i -\pd{}{x^i}\otimes dy^i
        =\begin{pmatrix}
        0& -1_n\\ 1_n&0
        \end{pmatrix},\\
        J&:=\pd{}{y^i}\otimes dx^i+\pd{}{x^i}\otimes dy^i
        =\begin{pmatrix}
        0& 1_n\\ 1_n&0
        \end{pmatrix},\\
        K&:=\pd{}{x^i}\otimes dx^i -\pd{}{y^i}\otimes dy^i
        =\begin{pmatrix}
        1_n& 0\\ 0&-1_n
        \end{pmatrix},\\
        h&:=\delta_{ij}(dx^i\otimes dx^j+dy^i\otimes dy^j)
        =\begin{pmatrix}
        1_n& 0\\ 0&1_n
        \end{pmatrix},\\
        k&:=\delta_{ij}(dx^i\otimes dy^j+dy^j\otimes dx^i)
        =\begin{pmatrix}
            0& 1_n\\ 1_n&0
        \end{pmatrix},\\
        \omega&:=\delta_{ij}(dx^i\otimes dy^j - dy^j\otimes dx^i)
        =\begin{pmatrix}
            0& 1_n\\ -1_n&0
        \end{pmatrix}.
    \end{align*}
    Here, the right-hand side denotes the matrix representation with respect
    to $(x^1,\ldots,x^n, y^1,\ldots,y^n)$.
\end{example}

\section{A Born structure on the tangent bundle of a Hessian manifold}

\subsection{Almost Born structures on tangent bundles}
In this section, we confirm that a manifold equipped with an affine connection
induces an almost para-quaternionic structure on its tangent bundle. Furthermore,
if the manifold is also equipped with a Riemannian metric, then it induces an almost
Born structure.\\
In the following, let~$M$ be an~$n$-dimensional manifold, 
$\nabla$ an affine connection on~$M$, 
and~$\pi:TM\rightarrow M$ the tangent bundle of~$M$.\\

While not explicitly stated in this form,
the following fact is proved in \cite{MS19}.
\begin{fact}\label{exm:stat}
    The (1,1)-tensor fields~$I,J$ and $K$ locally defined as follows,
    satisfy the compatibility conditions between charts.
    Thus, we obtain well-defined (1,1)-tensor fields on the total space of
    the tangent bundle~$TM$.
    Moreover, the triple~$(I,J,K)$ defines an almost para-quaternionic structure
    on~$TM$.\\

    Let~$(U;~\theta^1,\ldots,\theta^n)$ be an arbitrary chart on~$M$, and 
    $(TU;~x^1,\ldots,x^n,y^1,\ldots,y^n)$ a chart on~$TM$ defined as
    \begin{equation*}
        x^i:= \theta^i\circ\pi,\quad 
        y^i:= d\theta^i.
    \end{equation*} 
    Furthermore, we define a local frame~$(TU;~H_1,\ldots,H_n,V_1,\ldots,V_n)$
    as 
    \begin{equation*}
        H_i:=\pd{}{x^i}-(\Gamma^k_{ij}\circ \pi)y^j\pd{}{y^k},\quad 
        V_i:=\pd{}{y^i},
    \end{equation*}
    where $\Gamma_{ij}^k$ denotes the connection coefficients of~$\nabla$
    in the chart~$(U;~\theta^1,\ldots,\theta^n)$.
    Then, by using the dual frame of~$(TU;~H_1,\ldots,H_n,V_1,\ldots,V_n)$ i.e.
    \begin{equation*}
        {H^*}^i=dx^i,\quad
        {V^*}^i=(\Gamma^i_{jk}\circ \pi)y^kdx^j+dy^i,
    \end{equation*}
    we define the~$(1,1)$-tensor fields $I,J$ and~$K$ as
    \begin{align*}
        I&:=V_i\otimes H^{*i} -H_i\otimes V^{*i}
        =\begin{pmatrix}
            0& -1_n\\ 1_n&0
        \end{pmatrix},\\
        J&:=V_i\otimes H^{*i}+H_i\otimes V^{*i}
        =\begin{pmatrix}
            0& 1_n\\ 1_n&0
        \end{pmatrix},\\
        K&:=H_i\otimes H^{*i} -V_i\otimes V^{*i}
        =\begin{pmatrix}
            1_n& 0\\ 0&-1_n
        \end{pmatrix}.
    \end{align*}
    Here, the right-hand side denotes the matrix representation with respect
    to $(H_1,\ldots,H_n,V_1,\ldots,V_n)$.
\end{fact}
\begin{fact}[\cite{MS19}]
    Let~$g$ be a Riemannian metric on~$M$.
    We can construct the well-defined (0,2)-tensor fields~$h,k$ and~$\omega$
    as follows, using the local frame $(TU;~H_1,\ldots,H_n,V_1,\ldots,V_n)$ defined in
    Fact~\ref{exm:stat};
    \begin{align*}
        h&:=(g_{ij}\circ\pi)(H^{*i}\otimes H^{*j}+V^{*i}\otimes V^{*j})
        =\begin{pmatrix}
            G& 0\\ 0&G
        \end{pmatrix},\\
        k&:=(g_{ij}\circ\pi)(H^{*i}\otimes V^{*j}+V^{*J}\otimes H^{*i})
        =\begin{pmatrix}
            0& G\\ G&0
        \end{pmatrix},\\
        \omega&:=(g_{ij}\circ\pi)(H^{*i}\otimes V^{*j}-V^{*J}\otimes H^{*i})
        =\begin{pmatrix}
            0& G\\ -G&0
        \end{pmatrix},
    \end{align*}
    where~$G=(g_{ij}\circ\pi)_{i,j=1}^n$.
    Then, the tuple~$(I,J,K,h,k,\omega)$ defines an almost Born structure
    on~$TM$.
\end{fact}
\begin{note}\label{not:HesInt}
    If~$\nabla$ is flat, and
    $(U;\theta^1,\ldots,\theta^n)$ is an affine chart for~$\nabla$, 
    then we have:
    \begin{align*}
        I&:=\pd{}{y^i}\otimes dx^i -\pd{}{x^i}\otimes dy^i
        =\begin{pmatrix}
            0& -1_n\\ 1_n&0
        \end{pmatrix},\\
        J&:=\pd{}{y^i}\otimes dx^i+\pd{}{x^i}\otimes dy^i
        =\begin{pmatrix}
            0& 1_n\\ 1_n&0
        \end{pmatrix},\\
        K&:=\pd{}{x^i}\otimes dx^i -\pd{}{y^i}\otimes dy^i
        =\begin{pmatrix}
            1_n& 0\\ 0&-1_n
        \end{pmatrix},\\
        h&:=(g_{ij}\circ\pi)(dx^i\otimes dx^j+dy^i\otimes dy^j)
        =\begin{pmatrix}
            G& 0\\ 0&G
        \end{pmatrix},\\
        k&:=(g_{ij}\circ\pi)(dx^i\otimes dy^j+dy^j\otimes dx^i)
        =\begin{pmatrix}
            0& G\\ G&0
        \end{pmatrix},\\
        \omega&:=(g_{ij}\circ\pi)(dx^i\otimes dy^j-dy^j\otimes dx^i)
        =\begin{pmatrix}
            0& G\\ -G&0
        \end{pmatrix}.
    \end{align*}
\end{note}
\begin{example}
    Regard~$\Rn{n}$ as a Hessian manifold with the standard Hessian structure
    defined in Example~\ref{exm:StaHes}.
    In this case, the induced Born structure on~$T\Rn{n}$ coincides with
    the standard strongly integrable Born structure on~$\Cn{n}$ defined in
    Example~\ref{exm:StaBor}, through the following natural identification.
    \begin{equation*}
        T\Rn{n} \rightarrow \Cn{n}~:~(x,v)\mapsto x+\iu v
    \end{equation*}
\end{example}



\subsection{Integrability of the Born structures on tangent bundles}
In this section, we prove the main theorem that relates
the flatness conditions of a manifold equipped with an affine connection 
and a Riemannian metric, to the integrability conditions of 
the induced almost para-quaternionic structure or almost 
Born structure on the tangent bundle.\\
In the following, let~$M$ be an~$n$-dimensional manifold, 
$\nabla$ an affine connection on~$M$, $g$ a
Riemannian metric on~$M$, and~$\pi: TM\rightarrow M$ 
the tangent bundle of~$M$. 
\begin{proposition}\label{fac:int}
    For~$(M, \nabla, g)$ and the induced almost Born manifold~
    $(TM, I,J,K,h,k,\omega)$, the following statements hold.\\
    (1)~The following conditions are equivalent.\\
    \quad(1-a) $\nabla$ is flat. \\
    \quad(1-b) $I$ is integrable.\\
    \quad(1-c) $J$ is integrable.\\
    (2)~The curvature of~$\nabla$ vanishes~$\Leftrightarrow$~$K$ is integrable.\\
    (3)~$\nabla^*$ is torsion-free 
    $\Leftrightarrow$ $d\omega=0.$
\end{proposition}
\begin{proof}
    The implication (1-a)$\Leftrightarrow$(1-b) is proved 
    in \cite[\textsection 5. Corollary]{Dom62}.\\
    The statement (2) is proved in \cite[Lemma 3]{FRS17}.\\
    The statement (3) is proved in \cite[Theorem 1.1]{Sat07}.\\
    We now prove the implication (1-a)$\Leftrightarrow$(1-c).\\
    For the local frame~$(TU;~H_1,\ldots,H_n,V_1,\ldots,V_n)$
    defined in Fact~\ref{exm:stat}, the identities:
    \begin{equation*}
        [H_i,H_j]=-(R_{ijk}^l\circ \pi)y^kV_l,\quad
        [V_i,V_j]=0,\quad
        [H_i,V_j]=-\Gamma_{ij}^kV_k
    \end{equation*}
    hold. Then we have the following identities, from which the claim follows:
    \begin{align*}
        N_J(H_i,H_j)&=T_{ij}^kH_k-R_{ijk}^ly^kV_l,\\
        N_J(V_i,V_j)&=T_{ij}^kH_k-R_{ijk}^ly^kV_l,\\
        N_J(H_i,V_j)&=R_{ijk}^ly^kH_l-T_{ij}^kV_k.
    \end{align*}
\end{proof}
We note that the above proof of the implication 
(1-a)$\Leftrightarrow$(1-c) is the same as that in \cite{Dom62}
although the notation is different.
\ThmHesBor

\begin{proof}
    (1)~The implication~(1-a)$\Rightarrow$(1-c) follows
    from the expression in Note~\ref{not:HesInt}, 
    (1-c)$\Rightarrow$(1-b) follows from
    Note \ref{not:PQint}, and
    (1-b)$\Rightarrow$(1-a) follows from Proposition~\ref{fac:int}.\\
    (2)~The implication~(2-a)$\Rightarrow$(2-c) follows from 
    (1-a)$\Rightarrow$(1-c), Fact~\ref{not:HesTor}, and~Proposition~\ref{fac:int}.
    The implication (2-c)$\Rightarrow$(2-b) follows from the definition
    of the integrability conditions for Born structures.
    Furthermore, the implication (2-b)$\Rightarrow$(2-a)
    follows from Fact~\ref{not:HesTor} and 
    Proposition~\ref{fac:int}.
\end{proof}


\begin{thebibliography}{111111}
    \bibitem[Ama85]{Ama85}
    Amari, S. \textit{Differential-Geometrical Methods in Statistics.} 
    Berlin : Springer-Verlag, 1985.
    \bibitem[And05]{And05}
    Andrada, A. ``Complex product structures 
    and affine foliations.'' 
    \textit{Annals of Global Analysis and Geometry}, 27, 2005, 
    pp.377-405.
    \bibitem[CFG96]{CFG96}
    Cruceanu, V.,  Fortuny, P., 
    and Gadea, P. M. 
    ``A survey on paracomplex geometry.'' 
    \textit{The Rocky Mountain Journal of Mathematics}, 
    26(1), 1996, pp.83-115.
    \bibitem[CMMS04]{CMMS04}
    Cortés, V., Mayer, C., Mohaupt, T., and Saueressig, F. 
    ``Special geometry of euclidean supersymmetry 1. Vector multiplets.'' 
    \textit{Journal of High Energy Physics}, 2004(03)028, 2004.
    \bibitem[Dom62]{Dom62}
    Dombrowski, P. ``On the Geometry of the Tangent Bundle.'' 
    \textit{Journal für die reine und angewandte Mathematik}, 210, 1962, pp.73-88.
    \bibitem[FLM14]{FLM14}
    Freidel, L., Leigh, R. G., and Minic, D. 
    ``Born reciprocity in string theory and the nature of spacetime.'' 
    \textit{Physics Letters B}, 730, 2014, pp.302-306.
    \bibitem[FRS17]{FRS17}
    Freidel, L., Rudolph, F. J., and Svoboda, D. 
    ``Generalised kinematics for double field theory.'' 
    \textit{Journal of High Energy Physics}, 11(2017)175, 2017.
    \bibitem[FRS19]{FRS19}
    Freidel, L., Rudolph, F. J., and Svoboda, D. 
    ``A unique connection for Born geometry.'' 
    \textit{Communications in Mathematical Physics}, 372, 2019, pp.119-150.
    \bibitem[Fuj21]{Fuj21}
    Fujiwara, A. \textit{Foundations in Information Geometry.} (in Japanese)
    Tokyo : Kyoritsu Shuppan Co., Ltd., 2021.
    \bibitem[HKP25]{HKP25}
    Hamilton, M. J. D., Kotschick, D., and Pilatus, P. N. 
    ``Born geometry via K\" unneth structures 
    and recursion operators.'' 
    \textit{Communications in Mathematical Physics}, 406:123, 2025.
    \bibitem[IZ05]{IZ05}
    Ivanov, S. and Zamkovoy, S. 
    ``Parahermitian and paraquaternionic manifolds.''
    \textit{Differential Geometry and its Applications}, 23, 2005, pp.205-234.
    \bibitem[MS19]{MS19}
    Marotta, V. E. and Szabo, R. J. 
    ``Para‐Hermitian Geometry, Dualities and 
    Generalized Flux Backgrounds.'' 
    \textit{Fortschritte der Physik}, 67:1800093, 2019.
    \bibitem[Sat07]{Sat07}
    Satoh, H. 
    ``Almost Hermitian structures on tangent bundles.''
    arXiv preprint, arXiv:1908.10824, 2019.
    \bibitem[Shi76]{Shi76}
    Shima, H. ``On certain locally flat homogeneous manifolds 
    of solvable Lie groups.'' 
    \textit{Osaka Journal of Mathematics}, 13, 1976, pp.213-229.
    \bibitem[Sil08]{Sil08}
    Cannas da Silva, A. 
    \textit{Lectures on symplectic geometry.} 
    Berlin : Springer-Verlag, 2008.
    \bibitem[YA73]{YA73}
    Yano, K. and Ako, M. 
    ``Almost quaternion structures of the second kind 
    and almost tangent structures.'' 
    \textit{Kodai Mathematical Seminar Reports}, 25, 1973, pp.63-94.
\end{thebibliography}
\end{document}